\documentclass[letterpaper, 10 pt, conference]{ieeeconf}  

\IEEEoverridecommandlockouts                              

\usepackage{hyperref}%
\hypersetup{colorlinks=true, linkcolor=blue, breaklinks=true, urlcolor=blue}
\usepackage{doi}

\usepackage[all]{xy}
\usepackage{amsmath}
\usepackage{amscd}
\usepackage{amsthm}
\usepackage{mathrsfs}
\usepackage{amssymb}
\usepackage{tikz}
\usepackage[utf8]{inputenc}
\usepackage[yyyymmdd]{datetime}
\usepackage{graphicx}
\usepackage{subcaption}
\usepackage{lipsum}
\usepackage{multicol} \usepackage{mathtools}
\usepackage{multirow}
\usepackage{threeparttable}
\usepackage{pdfpages}
\usepackage{stmaryrd}
\let\labelindent\relax
 \usepackage{enumitem}

\usepackage{multicol} \usepackage{mathtools}
\usepackage{version,xspace}
\usepackage{url,doi}

\newcommand{\subscr}[2]{#1_{\textup{#2}}}
\newcommand{\supscr}[2]{#1^{\textup{#2}}} \newcommand{\setdef}[2]{\{#1
	\; | \; #2\}}

\newcommand{\KM}{Krasnosel’skii–Mann\xspace}

\newcommand\oprocendsymbol{\hbox{$\triangle$}}
\newcommand\oprocend{\relax\ifmmode\else\unskip\hfill\fi\oprocendsymbol}

\renewcommand{\theenumi}{(\roman{enumi})}

\DeclareSymbolFont{bbold}{U}{bbold}{m}{n}
\DeclareSymbolFontAlphabet{\mathbbold}{bbold}
\newcommand{\vect}[1]{\mathbbold{#1}}

\newcommand{\real}{\mathbb{R}}

\DeclareMathOperator{\Img}{\mathrm{Img}}
\renewcommand{\top}{\mathsf{T}} 

\newtheorem{theorem}{Theorem}
\newtheorem{proposition}[theorem]{Proposition}

\newtheorem{lemma}[theorem]{Lemma}
{\theoremstyle{definition}
	\newtheorem{definition}[theorem]{Definition}
	\newtheorem{remark}[theorem]{Remark}
	\newtheorem{example}[theorem]{Example}
	
        \newtheorem{assumption}[theorem]{Assumption}
	\newtheorem*{example*}{Example}
}

\newcommand{\suchthat}{\;\ifnum\currentgrouptype=16 \middle\fi|\;}
\newcommand{\scirc}{\raise1pt\hbox{$\,\scriptstyle\circ\,$}}

\title{\LARGE \bf
Resilience of Input Metering in Dynamic Flow Networks
}

\author{Saber Jafarpour$^{*}$ and Samuel Coogan$^{*}$
\thanks{$^{*}$Saber Jafarpour and Samuel Coogan are with the School of Electrical and Computer Engineering, Georgia Institute of Technology,
        {\tt\small \{saber,sam.coogan\}@}gatech.edu}.%
}

\begin{document}

\maketitle
\thispagestyle{empty}
\pagestyle{empty}

\begin{abstract}
In this paper, we study robustness of  input metering
policies in dynamic flow networks in the presence of
transient disturbances and attacks. We consider a
compartmental model for dynamic flow networks with a
First-In-First-Out (FIFO) routing rule as found in, \emph{e.g.}, transportation networks. We model
the effect of the transient disturbance as an abrupt change
to the state of the network and use the notion
of the region of attraction to measure the resilience of
the network to these changes.
For constant and periodic input metering, we introduce the notion of
monotone-invariant points to establish inner-estimates for the regions of attraction of
free-flow equilibrium points and free-flow periodic orbits using monotone systems theory. These results are applicable to, \emph{e.g.}, networks with cycles, which have not been considered in prior literature on dynamic flow networks with FIFO routing.
Finally, we propose two approaches for finding suitable
monotone-invariant points in the flow networks with FIFO rules. 
\end{abstract}

\section{INTRODUCTION}

\paragraph*{Problem statement and motivation}

Dynamic flow networks are a class of dynamical systems that models the flow of a commodity through a network of interconnected components. This modeling paradigm has been successfully used to study a wide range of natural and engineered systems including transportation networks, drinking water and irrigation networks, supply chain networks, and power grids. The dynamics of these networks are described by the rate of change of density of each compartment together with a policy describing routing of the flows among the connected compartments.


One of the important features of dynamic flow networks is the nonlinear behavior of their components. In particular, it is common for the flow throughput of these networks to increase until a critical capacity and then the network enters a congested regime in which the throughput decreases and can drop to zero. This phenomenon, which is usually referred to as congestion, can propagate through the network and cause a cascading failure of components. In order to mitigate the effects of congestion and restore the full utilization of the flow network, different control strategies have been proposed in the literature. Arguably, one of the most widely-applicable and efficient control strategy for congestion mitigation is input metering where the optimal throughput is obtained by restricting the amount of the input flows in certain components of the system.    


In this work, we focus on the effect of transient disturbance and attack on input metering strategies. Several notions have been proposed in the literature to measure the performance of input metering in flow network systems. However, these measures are either only applicable to static networks~\cite{PE-AF-CS:1956} or they ignore the effect of transient disturbances and attacks in the system~\cite{GC:17}. In contrast, in this work, we model the effect of a transient disturbance as a change in initial distribution of commodity in flow networks. In this case, the region of attraction of the dynamical system can be used to measure the resilience of the input metering to disturbances.   

\paragraph*{Literature review}

The use of compartmental models for studying dynamic flow networks has a rich history~\cite{JAJ-CPS:93}. For transportation networks, the cell transmission model has been extensively used to study dynamic behavior of vehicles in roads~\cite{CFD:94,CFD:95,GG-RH-AAK-PV-JK:08}. Monotone system theory and contraction theory are two of the most prominent tools in studying dynamic behaviors of flow networks~\cite{GC:17}. The papers~\cite{GC-KS-DA-MAD-EF:13,GC-KS-DA-MD-EF:13b} study the throughput of dynamic flow networks and propose a robust routing policy to ensure the monotonicity of the closed-loop system. For transportation network,~\cite{EL-GC-KS:14} shows that certain classes of flow networks with non-FIFO rules are monotone and studies their dynamic stability using contraction theory. While dynamic analysis of flow networks with non-FIFO policies has recently gained much attention, for many important classes of flow networks, the FIFO routing rule is considered to be a more realistic modeling assumption. It is known that the FIFO routing rule can lead to flow dynamics that are not cooperative and monotone theory is not applicable on the whole network domain~\cite{SC-MA:15}. Extensions of monotone system theory have been proposed to study stability of flow networks with FIFO routing rules~\cite{SC-MA:16}. Input metering strategies have been proposed in the literature to optimize the throughput of the flow networks as well as to mitigate the effect of congestion. Input metering in traffic networks using model predictive control is considered in~\cite{MS-CR-JL:17}. Ramp metering has also been used in traffic networks as a type of input metering, and we refer to the survey~\cite{MP-AK:02}.


\paragraph*{Contributions}
In this paper, we provide a framework to study performance of the input metering strategies in the flow networks with respect to the transient disturbances and attacks. We consider a class of network flow dynamics where the rate of change of density is determined by the difference between the inflow and outflow and routing of flows is governed by a FIFO rule. The trace of the transient disturbances on flow networks is modeled by an abrupt change in the initial densities and the robustness of the input metering strategy is measured using the regions of attraction. By introducing the notion of a monotone-invariant point, we establish a framework to employ monotone system theory in stability analysis of flow networks with FIFO rules. As the first contribution of this paper, we characterize the existence and local stability of the free-flow equilibrium points and free-flow periodic orbits of the dynamic flow networks with FIFO rules. Regarding the local stability of the free-flow equilibrium points, our framework extends the existing results in the literature to cyclic flow networks. Moreover, for periodic input metering, our result on the existence and local stability of free-flow periodic orbits is novel. As the main contribution of this paper, we use monotone-invariant points to provide inner-estimates on the regions of attraction of equilibrium points and periodic orbits of the flow networks. Our inner-estimates of regions of attraction are (i) sharper than the existing estimates in the literature, (ii) applicable to flow networks with cyclic topology, and (iii) useful for investigating transient stability of periodic orbits. Finally, we provide an analytic method and an iterative approach for finding suitable monotone-invariant points in the dynamic flow networks. 


\section{Notation and Mathematical Preliminary}\label{sec:notation}
For every $x,y\in \real^n$, we write $x\le y$ if $x_i\le y_i$, for every $i\in \{1,\ldots,n\}$. For every $x,y\in \real^n$ such that $x\le y$, we define the box $[x,y]=\setdef{z\in \real^n}{x\le z\le y}$. For a set $S\subseteq \real^n$, the interior and closure of $S$ are denoted by $\mathrm{int}(S)$ and $\mathrm{cl}(S)$, respectively. We denote the $\ell_p$-norm on $\real^n$ by $\|\cdot\|_{p}$. For a given norm $\|\cdot\|$ on $\real^n$, the induced norm $\|\cdot\|_i$ on $\real^{n\times n}$ is defined by $\|A\|_i=\sup_{x\ne 0}\frac{\|Ax\|}{\|x\|}$. For a matrix $A\in \real^{n\times n}$ and a norm $\|\cdot\|$ on $\real^n$, the matrix measure of $A$ with respect to $\|\cdot\|$ is $\mu_{\|\cdot\|}(A) = \lim_{h\to 0^+} \frac{\|I_n+hA\|_i-1}{h}$.
Consider the following dynamical system on $\real^n$
\begin{align}\label{eq:dynamicalsystem}
  \dot{x}=f(t,x).
\end{align}
The flow of~\eqref{eq:dynamicalsystem} at time $t$ starting from $x_0$ at time $t_0$ is denoted by $\phi^f(t,t_0,x_0)$. Given a norm $\|\cdot\|$ on $\real^{n}$, the dynamical system~\eqref{eq:dynamicalsystem} is contracting with respect to $\|\cdot\|$ if, there exists $c>0$ such that for every $x_0,y_0 \in \real^n$ and every $t\ge t_0$, 
\begin{align*}
  \|\phi^f(t,t_0,x_0)- \phi^f(t,t_0,y_0)\| \le e^{-c(t-t_0)} \|x_0-y_0\|.
\end{align*}
and is weakly contracting if, for every $x_0,y_0 \in \real^n$ and every $t\ge t_0$, we have $\|\phi^f(t,t_0,x_0)- \phi^f(t,t_0,y_0)\| \le \|x_0-y_0\|$.

We model a dynamic flow network as a directed graph $G=(V,\mathcal{O})$ where $V$ are the nodes and $\mathcal{O}$ are the directed links connecting the nodes. A set of entry links
$\mathcal{R}$ allows exogenous flow to enter the network, and the set of all links in the network is denoted by $\mathcal{L}=\mathcal{R}\cup\mathcal{O}$. For every link $l\in \mathcal{O}\cup \mathcal{R}$, the head and the tail of the link are denoted by $\sigma(l)$ and $\tau(l)$,
respectively. By convention, we assume that, on a link $l$, the commodity flows from tail $\tau(l)$ to the head $\sigma(l)$ and we have $\tau(l)=\emptyset$ for every $l\in \mathcal{R}$. For every $v\in V$, we
denote the set of input (resp. output) links to node $v$ by
$\supscr{\mathcal{L}}{in}_v$ (resp. $\supscr{\mathcal{L}}{out}_v$). More precisely, 
\begin{align*}
  \supscr{\mathcal{L}}{in}_v &= \setdef{i\in \mathcal{L}}{\sigma(i)=v},\quad \supscr{\mathcal{L}}{out}_v=\setdef{i\in \mathcal{L}}{\tau(i)=v}.
\end{align*}
A node $v\in V$ is a diverging node if $|\supscr{\mathcal{L}}{out}_v|>1$. The set of diverging nodes is denoted by $\supscr{V}{div}$. We assume that
for each node $v\in V$, there exists a set of fixed split
ratios $\{R^v_{i}\}_{i\in\mathcal{L}}$ such that $R^v_{i}>0$
for every $i\in \supscr{\mathcal{L}}{out}_v$ and $\sum_{i\in \supscr{\mathcal{L}}{out}_v}R^v_{i} \le 1$.
The strict inequality above can happen if a fraction of flow at node $v$ is going
out of the network. We define the set of out-nodes by $\supscr{V}{out}= \setdef{v\in V}{\sum_{i\in\supscr{\mathcal{L}}{out}_v}R^v_{i} < 1}$ and the set of in-nodes by $\supscr{V}{in}= \setdef{v\in V}{\mathcal{R}\cap \mathcal{L}^{\mathrm{in}}_v\ne\emptyset}$. For every $i\in \mathcal{L}$, the dynamic flow network satisfies
\begin{align}\label{eq:traffic}
  \dot{x}_i = \supscr{f}{in}_i(x,u) -
  \supscr{f}{out}_i(x) : = F_{i}(x,u),
\end{align}
where $x_i$ is the density of the commodity at link $i$ and $x=(x_1,\ldots,x_{|\mathcal{L}|})^{\top}\in \real^{|\mathcal{L}|}$. The functions $\supscr{f}{in}_i$
and $\supscr{f}{out}_i$ are the inflow and outflow to the link $i\in
\mathcal{L}$, and $u\in \real^{|\mathcal{R}|}$ is the input metering at the entry links. We assume that, every link $i\in \mathcal{L}$ can accommodate a maximum density denoted by $\overline{x}_i$, has a supply function $s_i:[0,\overline{x}_{i}]\to \real_{\ge 0}$, and has
a demand function $d_i: [0,\overline{x}_{i}]\to \real_{\ge 0}$ such that
\begin{enumerate}
\item\label{p1} $d_i$ is strictly increasing, Lipschitz continuous, and piecewise real analytic with $d_i(0)=0$;
\item\label{p2} $s_i$ is strictly decreasing, Lipschitz continuous, and piecewise real analytic with $s_i(\overline{x}_{i})=0$.
\end{enumerate}
Because of properties~\ref{p1} and ~\ref{p2}, for every $i\in
\mathcal{L}$, there exists a critical density $\supscr{x}{crit}_i$ such that
$d_i(\supscr{x}{crit}_i) = s_i(\supscr{x}{crit}_i)$. We also define
the critical flow at link $i\in \mathcal{L}$ by
$\supscr{f}{crit}_i := d_i(\supscr{x}{crit}_i) =
s_i(\supscr{x}{crit}_i)$. We assume that the flows in the network are routed based on a FIFO rule. Following~\cite{AAK-PV:10}, we adopt the following FIFO rule at each node $v\in V$:
\begin{align}\label{eq:FIFO-traffic}
  \supscr{f}{out}_i(x) &= \alpha^{\sigma(i)}(x)d_i(x_i),\nonumber\\
  \alpha^{v}(x) &= \min_{i\in \supscr{\mathcal{L}}{out}_v}\left\{1,
  \frac{s_i(x_i)}{R^{v}_{i}\sum_{j\in
                  \supscr{\mathcal{L}}{in}_v}d_j(x_j)}\right\}, \nonumber\\
                  \supscr{f}{in}_i(x,u) &= \begin{cases}
                    \min\{u_i, s_i(x_i)\}, & i\in \mathcal{R}\\
                    R^{\tau(i)}_{i} \sum_{j\in \supscr{\mathcal{L}}{in}_{\tau(i)}}
                    \supscr{f}{out}_j(x_j), & i\in \mathcal{O}. 
                    \end{cases}
\end{align}

\section{Monotone-flow domains and monotone-invariant points}\label{sec:mon}

In this section, we investigate the applicability of monotone system theory in flow networks with FIFO rules. We introduce the monotone-flow domain as an extension of the free-flow domain and show that the flow network~\eqref{eq:traffic} and~\eqref{eq:FIFO-traffic} is a monotone dynamical system on this domain. 

\begin{definition}[Free-flow and monotone-flow domains]\label{def:monotoneflow}
  Consider the dynamic flow system~\eqref{eq:traffic} and~\eqref{eq:FIFO-traffic}. For the density vector $x\in \real^{|\mathcal{L}|}$ and link $i\in \mathcal{L}$, we say that $i$ is 
  \begin{enumerate}
      \item in \emph{free-flow} if $\supscr{f}{out}_i(x)=d_i(x_i)$;
      \item in \emph{congestion} if $\supscr{f}{out}_i(x)< d_i(x_i)$.
  \end{enumerate}
  Additionally, 
  \begin{enumerate}\setcounter{enumi}{2}
      \item the \emph{free-flow domain} $\mathcal{F}\subseteq \real^{|\mathcal{L}|}$ is defined by 
      \begin{align*}
  \mathcal{F} =\{x\in
  \real^{|\mathcal{L}|}_{\ge
  0}  \mid  \supscr{f}{out}_i(x) =d_i(x_i), \mbox{ for }i\in \mathcal{L}\}. 
\end{align*}
      \item the \emph{monotone-flow domain} $\mathcal{M}\subseteq \real^{|\mathcal{L}|}$ is defined by 
      \begin{align*}
  \mathcal{M} =\{x\in
  \real^{|\mathcal{L}|}_{\ge
  0}  \mid & \supscr{f}{out}_i(x) =d_i(x_i), \nonumber\\ & \mbox{ for }i\in \supscr{\mathcal{L}}{out}_v \mbox{ with } v\in \supscr{V}{div}\}. 
\end{align*}
  \end{enumerate}
  \end{definition}
Intuitively, $\mathcal{F}$ is the set of density vectors for which every link is in free-flow and $\mathcal{M}$ is the set of density vectors for which the out-links of diverging nodes are in free-flow. Therefore, it is easy to see that $\mathcal{F}\subseteq \mathcal{M}$. It is important to note that the free flow domain $\mathcal{F}$ and the monotone-flow domain $\mathcal{M}$ are independent of the input signal $u$. The following proposition studies properties of the domain $\mathcal{M}$. We refer to Appendix~\ref{app:monotone-flow} for the proof.

\begin{proposition}[Properties of the monotone-flow domain]\label{thm:free-flow-property}
  Consider the dynamic flow network~\eqref{eq:traffic} and~\eqref{eq:FIFO-traffic}. The following statements hold:
  \begin{enumerate}

  \item\label{p1:cone} for every $x\in\mathcal{M}$, we have $[\vect{0}_{|\mathcal{L}|},x]\subseteq \mathcal{M}$. 
   \item\label{p2:weak-contraction} the dynamical system is monotone and weakly contracting on $\mathcal{M}$
  with respect to $\ell_1$-norm.
    \end{enumerate}
\end{proposition}

Next, we introduce the notion of monotone-invariant points in the monotone-flow domain $\mathcal{M}$. As we see in Section~\ref{sec:eq}, this notion plays a crucial role in our approach for inner-estimation of regions of attraction. 

\begin{definition}[Monotone-invariant points]
  Consider the dynamical flow network~\eqref{eq:traffic} with the
  FIFO rule~\eqref{eq:FIFO-traffic} with an input metering $u:\real_{\ge 0}\to \real^{|\mathcal{R}|}$. A point $y\in \mathcal{M}$ is a \emph{monotone-invariant point} if the trajectory of the system starting at $y$ remains inside $\mathcal{M}$. 
\end{definition}

We define the vector field $H:\real^{|\mathcal{L}|}\times \real^{|\mathcal{R}|}\to \real^{|\mathcal{L}|}$ by $H_i(x,u) = \supscr{h}{in}_i(x,u) - \supscr{h}{out}_i(x)$ where:
\begin{align}\label{eq:H}
  \supscr{h}{out}_i(x) &=\begin{cases}
  d_i(x_i), & i\in \supscr{\mathcal{L}}{out}_v,\; v\in \supscr{V}{div}\\
    \supscr{f}{out}_i(x), & \mbox{otherwise.} 
  \end{cases}\nonumber \\ 
                  \supscr{h}{in}_i(x,u) &= \begin{cases}
                    \min\{u_i, s_i(x_i)\}, & i\in \mathcal{R}\\
                    R^{\tau(i)}_{i} \sum_{j\in \supscr{\mathcal{L}}{in}_{\tau(i)}}
                    \supscr{h}{out}_j(x_j), & i\in \mathcal{O}. 
                    \end{cases}
\end{align}
The next proposition studies properties of the vector field $H$ and shows that $H$ can be considered as a \emph{monotone extension} of the flow network~\eqref{eq:traffic} and~\eqref{eq:FIFO-traffic} outside the monotone-flow domain $\mathcal{M}$. We refer to Appendix~\ref{sec:ext} for the proof. 

\begin{proposition}[Monotone extension]\label{thm:ext}
Consider the dynamical flow network~\eqref{eq:traffic} with the FIFO rule~\eqref{eq:FIFO-traffic} and let $H$ be the vector field defined in~\eqref{eq:H}. The following statements hold:
  \begin{enumerate}
      \item\label{p1:H} $H$ is monotone and weakly contracting with respect to the $\ell_1$-norm on $\real^{|\mathcal{L}|}_{\ge 0}$; 
      \item\label{p2:H} $H(x,u)=F(x,u)$, for every $x\in \mathcal{M}$.
  \end{enumerate}
\end{proposition}

\section{Regions of attraction for equilibrium points and periodic orbits of flow networks}\label{sec:eq}

Now, we use the notion of monotone-invariant points to present a framework for estimating regions of attraction of dynamic flow networks~\eqref{eq:traffic} and~\eqref{eq:FIFO-traffic} with constant and periodic input metering. For flow networks with constant input metering, we first obtain a closed-form expression for the non-congested equilibrium points of the system. We define
$R_{\mathcal{O}}\in \real^{|\mathcal{O}|\times |\mathcal{O}|}$ and $R_{\mathcal{R}}\in \real^{|\mathcal{O}|\times |\mathcal{R}|}$ as follow:
\begin{align*}
  [R_{\mathcal{O}}]_{kl} &=
  \begin{cases}
    R^{\sigma(l)}_k, & l,k\in\mathcal{O}\\
    0 & \mbox{otherwise.} 
  \end{cases}\\ 
[R_{\mathcal{R}}]_{kl} &=
  \begin{cases}
    R^{\sigma(l)}_k, & k\in\mathcal{O}, l\in \mathcal{R}\\
    0 & \mbox{otherwise.}
    \end{cases}
\end{align*}
In the rest of this paper, we assume that the flow networks satisfy the following assumption.
\begin{assumption}\label{asu:invertible}
In the flow network $G=(V,\mathcal{O})$, for every $v\in V$, 
\begin{enumerate}
    \item there exists a directed path from $v$ to a node $w\in V^{\mathrm{out}}$,
    \item there exists a directed path from a node $w\in V^{\mathrm{in}}$ to $v$.
\end{enumerate}
\end{assumption}
Assumption~\ref{asu:invertible} requires that, for every node in the network, there is a directed path to output flows and there is a directed path to an input flow. Note that $I-R_{\mathcal{O}}$ is a compartmental matrix. Thus, using Assumption~\ref{asu:invertible}, we can conclude that the matrix $I-R_{\mathcal{O}}$ is Metzler, Hurwitz, and invertible~\cite[Corollary 4.11]{FB:21}. Under Assumption~\ref{asu:invertible}, one can define $P = (I-R_{\mathcal{O}})^{-1}R_{\mathcal{R}}$ and show that $P\ge 0$~\cite[Theorem 9.5]{FB:21}. For a given input
$u\in \real^{|\mathcal{R}|}_{\ge 0}$, we define
\begin{align}\label{eq:flow-formula}
  f^e_i(u) = \begin{cases}
    u_i, & i\in \mathcal{R}, \\
    [Pu]_i, & i\in \mathcal{O}.
  \end{cases}
\end{align}
We also define the density vector $x^e(u)\in \real^{|\mathcal{L}|}$ as follows:
\begin{align}\label{eq:eqpt-formula}
  x^e_{i}(u) = d_i^{-1}(f^e_i(u)), \qquad\mbox{for every }i\in
               \mathcal{L}. 
\end{align} 

Next, we introduce two suitable classes of input metering. 

\begin{definition}[Feasible and strictly feasible input metering]
For the dynamical system~\eqref{eq:traffic} with the FIFO rule~\eqref{eq:FIFO-traffic},
\begin{enumerate}
    \item the set of \emph{feasible inputs}
$\mathcal{U}\subseteq \real^{|\mathcal{R}|}_{\ge 0}$ is:
\begin{align*}
  \mathcal{U} = \setdef{u\in \real^{|\mathcal{R}|}_{\ge 0}}{f^e_i(u)
  \le \supscr{f}{crit}_i, \mbox{ for all } i\in \mathcal{L}}.
\end{align*}
    \item the set of \emph{strictly feasible inputs} $\mathcal{U}_{\mathrm{str}}\subseteq \real^{|\mathcal{R}|}_{\ge 0}$ is:
\begin{align*}
  \mathcal{U}_{\mathrm{str}} = \setdef{u\in \real^{|\mathcal{R}|}_{\ge 0}}{f^e_i(u)
  <\supscr{f}{crit}_i, \mbox{ for all } i\in \mathcal{L}}.
\end{align*}
\end{enumerate}
\end{definition}
We also introduce the following assumption which requires the feasible input signal $u:\real_{\ge 0}\to \mathcal{U}$ to be upper-bounded by some $\overline{u}\in \mathcal{U}$.
\begin{assumption}\label{assum:good}
For the feasible input metering $u:\real_{\ge 0}\to \mathcal{U}$, there exists $\overline{u}\in \mathcal{U}$ such that $u(t)\le \overline{u}$, for every $t\in \real_{\ge 0}$.
\end{assumption}

Now, we can state the main results of this paper which provides inner-estimates for regions of attraction of free-flow equilibrium points of the flow networks with FIFO rules. 

\begin{theorem}[Regions of attractions for free-flow equilibrium points]\label{thm:equilibrium}
Consider the flow network~\eqref{eq:traffic} with the FIFO rule~\eqref{eq:FIFO-traffic} and a constant input metering $u\in \mathcal{U}$. Then, for $x^e(u)$ defined in~\eqref{eq:eqpt-formula} and for every monotone-invariant point $y\in \mathcal{M}$, the following statements hold: 
\begin{enumerate}
\item\label{p1:cyclic} $x^e(u)$ is the unique equilibrium point in $\mathcal{F}$; 
  \item\label{p3:converge} every trajectory starting in $[\vect{0}_{|\mathcal{L}|},y]$ converges to an equilibrium point $x^*\in \mathcal{M}$ with $x^*\ge x^e(u)$.
\end{enumerate}
Additionally, if $u\in\subscr{\mathcal{U}}{str}$, then
\begin{enumerate}\setcounter{enumi}{2}
\item\label{p2:cyclic} $x^e(u)$ is locally asymptotically stable and is the unique equilibrium point of the system in $\mathcal{M}$;
\item\label{p4:free-flow} every trajectory starting in $[\vect{0}_{|\mathcal{L}|},y]$ converges to $x^e(u)$.
\end{enumerate}
\end{theorem}
\begin{proof}
 Regarding part~\ref{p1:cyclic}, we need to show that, for every $i\in \mathcal{L}$, we have 
 \begin{align}\label{eq:eqpt-equality}
\supscr{f}{in}_i(x^e(u),u) = \supscr{f}{out}_i(x^e(u))      
 \end{align}
 Note that by definition of the feasible set $\mathcal{U}$, we have $x^e_i(u) = d_i^{-1}(f^e_i(u))\le \supscr{x}{crit}$, for every $i\in \mathcal{L}$. This implies that $f^e_i(u)\le s_i(x^e_i(u))$, for every $i\in \mathcal{L}$. We first show that $\alpha^v(x^e(u))=1$, for every $v\in V$. We compute
 \begin{align*}
     R^v_i\sum_{j\in \mathcal{L}^{\mathrm{in}}_v} d_j(x^e_j(u)) &=  R^v_i\sum_{j\in \mathcal{L}^{\mathrm{in}}_v} f_j^e(u)= f^e_i(u) \le s_i(x^e_i(u)),
 \end{align*}
 where the second inequality hold by definition of $f^e(u)$ in~\eqref{eq:flow-formula}. This implies that $\alpha^v(x^e(u))=1$. Using this observation, we have $\supscr{f}{out}_i(x^e(u)) = d_i(x^e_i(u)) = f^e_i(u)$, for every $i\in \mathcal{L}$. Moreover, for every $i\in \mathcal{O}$, we get $\supscr{f}{in}_i(x^e(u)) = R^{\tau(i)}_i \sum_{j\in\supscr{\mathcal{L}}{in}_{\tau(i)}} f^e_j(u)$ and, every $i\in \mathcal{R}$, we have  $\supscr{f}{in}_i(x^e(u)) = u_i$. Therefore $x^e(u)\in \mathcal{F}$ and equality~\eqref{eq:eqpt-equality} holds by definition of $f^e(u)$ in~\eqref{eq:flow-formula}.

   Regarding part~\ref{p3:converge}, consider the vector field $H$ defined by~\eqref{eq:H}. By Proposition~\ref{thm:ext}\ref{p2:H}, $H(x,u)=F(x,u)$, for every $x\in \mathcal{M}$. This implies that every trajectory of the system that remains in $\mathcal{M}$ is also a trajectory of $H$. For every $z\in [\vect{0}_{|\mathcal{L}|},y]$, we let $x_z:\real_{\ge 0}\to \real^{|\mathcal{L}|}$ be the trajectory of the system starting at $z$. Since $y$ is a monotone-invariant point, $x_y$ remains in $\mathcal{M}$ and therefore it is a trajectory of the vector field $H$. Moreover, by Proposition~\ref{thm:free-flow-property}, we have $[\vect{0}_{|\mathcal{L}|},y]\subseteq \mathcal{M}$ and the system is monotone on $\mathcal{M}$. This implies that, for every $p\in [\vect{0}_{|\mathcal{L}|},y]$, $x_p(t)\le x_y(t)$, for every $t\in \real_{\ge 0}$. As a result, for every $p\in [\vect{0}_{|\mathcal{L}|},y]$, the trajectory $x_p$ remains inside $\mathcal{M}$ and therefore is a trajectory of the vector field $H$. On the other hand, the vector field $H$ is piecewise real analytic and, by Proposition~\ref{thm:ext}\ref{p1:H}, it is weakly contracting with respect to $\ell_1$-norm on $\real^{|\mathcal{L}|}_{\ge 0}$. Therefore, by~\cite[Theorem 21]{SJ-PCV-FB:19q}, every trajectory of $H$ converges to an equilibrium point of $H$. This means that, for every $p\in [\vect{0}_{|\mathcal{L}|},y]$, the trajectory of the system starting at $p$ converges to an equilibrium point of the system in $\mathcal{M}$. Now, we show that every equilibrium point $x^*\in\mathcal{M}$ of the system satisfies $x^*\ge x^e(u)$. Suppose that this is not true, i.e., there exists an equilibrium point $x^*\in \mathcal{M}$ such that  $x^*\not\ge x^e(u)$. Since $F(x^*,u)=F(x^e(u),u)=\vect{0}_{|\mathcal{L}|}$, by Lemma~\ref{prop:invariant} the boxes $[\vect{0}_{|\mathcal{L}|},x^*]$ and $[\vect{0}_{|\mathcal{L}|},x^e(u)]$ are invariant sets. Therefore, $S=[\vect{0}_{|\mathcal{L}|},x^e(u)]\cap [\vect{0}_{|\mathcal{L}|},x^*]\subseteq \mathcal{F}$ is also an invariant set. Since $x^*\not\ge x^e(u)$ and $x^e(u)$ is the only equilibrium point of the system in $\mathcal{F}$, the set $S$ does not contain any equilibrium point of the system. On the other hand, the dynamical system is weakly contracting with respect to $\ell_1$-norm on the convex compact invariant set $S$. This is a contradiction, since by~\cite[Theorem 19]{SJ-PCV-FB:19q} the system have an equilibrium point in $S$. Therefore, for every equilibrium point $x^*\in \mathcal{M}$, we have $x^*\ge x^e(u)$.    
    
 Regarding part~\ref{p2:cyclic}, by definition of the feasible set $\subscr{\mathcal{U}}{str}$ and  by continuity of the supply and demand functions, there exists $\overline{u}\in \subscr{\mathcal{U}}{str}$ such that $u<\overline{u}$. By Lemma~\ref{lem:monotone-eqpt}, we have $x^e(u)<x^e(\overline{u})$ and therefore $x^e(u) \in \mathrm{int}([\vect{0}_{|\mathcal{L}|},x^e(\overline{u})])$. Moreover,
\begin{align*}
    F(x^e(\overline{u}),u) \le F(x^e(\overline{u}),\overline{u})=\vect{0}_{|\mathcal{L}|}. 
\end{align*}
Therefore, by Lemma~\ref{prop:invariant}, the box $[\vect{0}_{|\mathcal{L}|},x^e(\overline{u})]$ is an invariant set for the system. Additionally, the vector field $F$ is piecewise real analytic and, by Proposition~\ref{thm:free-flow-property}\ref{p2:weak-contraction}, it is weakly contracting with respect to the $\ell_1$-norm on $[\vect{0}_{|\mathcal{L}|},x^e(\overline{u})]$. Therefore, we can use~\cite[Theorem 21]{SJ-PCV-FB:19q}, to show that every trajectory of the system starting in the box  $[\vect{0}_{|\mathcal{L}|},x^e(\overline{u})]$ converges to an equilibrium point in  $[\vect{0}_{|\mathcal{L}|},x^e(\overline{u})]$. By part~\ref{p2:cyclic} $x^e(u)$ is the unique equilibrium point of the system in the box $[\vect{0}_{|\mathcal{L}|},x^e(\overline{u})]\subseteq \mathcal{F}$. Thus, every trajectory of the system starting in $[\vect{0}_{|\mathcal{L}|},x^e(\overline{u})]$ converges to $x^e(u)$. This means that $x^e(u)$ is locally asymptotically stable. 

Now, we prove uniqueness of $x^e(u)$ in $\mathcal{M}$. Suppose that $x^*$ is another equilibrium point of the system in $\mathcal{M}$. By part~\ref{p3:converge}, we have $x^*\ge x^e(u)$. Since $F(x^*,u)=\vect{0}_{|\mathcal{L}|}$, by Lemma~\ref{prop:invariant}, the box $[\vect{0}_{|\mathcal{L}|},x^*]$ is an invariant set. Moreover, $x^e(u)$ is locally asymptotically stable. Therefore, by~\cite[Theorem 19]{SJ-PCV-FB:19q} every trajectory in $[\vect{0}_{|\mathcal{L}|},x^*]$ converges to $x^e(u)$. This implies that $x^*=x^e(u)$. 

Regarding part~\ref{p4:free-flow}, by part~\ref{p2:cyclic}, if $u\in \subscr{\mathcal{U}}{str}$, then $x^e(u)$ is the unique locally asymptotically stable equilibrium point of the system in $\mathcal{M}$. Therefore, by part~\ref{p3:converge}, every trajectory starting in $[\vect{0}_{|\mathcal{L}|},y]$ converges to $x^e(u)$.\end{proof}

\begin{remark}[Comparison with the literature]
The following remarks are in order.
\begin{enumerate}
    \item Theorem~\ref{thm:equilibrium}\ref{p3:converge} and~\ref{p4:free-flow} are novel and extend the existing inner-estimates for region of attraction in the literature (see \cite[Proposition 4]{SC-MA:15}) to domains larger than $[\vect{0}_{|\mathcal{L}|},x^e(u)]$ and to networks with arbitrary (cyclic) topology.
    \item Theorem~\ref{thm:equilibrium}\ref{p2:cyclic} extends the existing results about local asymptotic stability of $x^e(u)$ (see~\cite[Proposition 2]{SC-MA:15}) to networks with cyclic topology. 
    \item For a feasible input $u$ such that $u\not\in \subscr{\mathcal{U}}{str}$, the flow network~\eqref{eq:traffic} and~\eqref{eq:FIFO-traffic} can have an equilibrium point $x^*$ in the monotone-flow domain $\mathcal{M}$ which is not in the free-flow domain $\mathcal{F}$. Example~\ref{ex:1} shows that free-flow and non-free-flow equilibrium points can coexist in the flow networks with FIFO rules.
\end{enumerate}

\end{remark}

Next, we study the region of attraction of free-flow periodic orbits in dynamic flow networks.

\begin{theorem}[Regions of attraction for free-flow periodic orbits]\label{thm:orbit}
Consider the dynamic flow network \eqref{eq:traffic} with the FIFO rule~\eqref{eq:FIFO-traffic} and a periodic input metering $u:\real_{\ge 0}\to \mathcal{U}$ with period $T>0$, i.e., $u(t+T)=u(t)$, for every $t\in \real_{\ge 0}$. Assume that $u$ satisfies Assumption~\ref{assum:good}.
Then, the following statement holds: 
\begin{enumerate}
    \item\label{p1:period} there exists a periodic solution $\rho:\real_{\ge 0}\to \mathcal{F}$ with period $T$, i.e., $\rho(t+T)=\rho(t)$, for every $t\in \real_{\ge 0}$; 
    \end{enumerate}
    Additionally, if Assumption~\ref{assum:good} holds for some $\overline{u}\in \subscr{\mathcal{U}}{str}$, then
    \begin{enumerate}\setcounter{enumi}{1}
    \item\label{p2:acyclic-period} the periodic solution $\rho$ is locally asymptotically stable. 
    \item\label{p2:perodic-converge} for every monotone-invariant point $y\in \mathcal{M}$, any trajectory starting in $[\vect{0}_{|\mathcal{L}|},y]$ converges to the periodic solution $\rho$.
    \end{enumerate}
\end{theorem}
\begin{proof}
 First note that since the supply and demand functions are continuous, the set $\mathcal{U}$ is closed in $\real^{|\mathcal{R}|}$. Let $x^e(\overline{u})$ be as defined in~\eqref{eq:eqpt-formula}. By Proposition~\ref{thm:free-flow-property}\ref{p1:cone}, the box $[\vect{0}_{|\mathcal{L}|},x^e(\overline{u})]$ is inside the monotone-invariant domain $\mathcal{M}$ and, by Lemma~\ref{prop:invariant}, it is an invariant set for the dynamical system. Regarding part~\ref{p1:period}, define the map $G_T:[\vect{0}_{|\mathcal{L}|},x^e(\overline{u})]\to \real^{|\mathcal{L}|}$ by 
 \begin{align*}
     G_T(x)=\phi(T,0,x),\qquad\mbox{ for every }x\in [\vect{0}_{|\mathcal{L}|},x^e(\overline{u})],
 \end{align*}
 where $\phi(t,t_0,x_0)$ is the flow of the system $\dot{x}=F(x,u(t))$ at time $t$, starting from $x_0$ at time $t_0$. By continuous dependence of solutions of dynamical systems on their initial conditions~\cite[Chapter V, Theorem 2.1]{PH:02}, the map $G_T$ is continuous. Since the box $[\vect{0}_{|\mathcal{L}|},x^e(\overline{u})]$ is an invariant set, $\Img(G_T) \subseteq [\vect{0}_{|\mathcal{L}|},x^e(\overline{u})]$. Therefore, $G_T$ is a continuous map from the convex compact set $[\vect{0}_{|\mathcal{L}|},x^e(\overline{u})]$ into itself. Therefore, by the Brouwer's fixed-point theorem, $G_T$ has a fixed point $x^*\in [\vect{0}_{|\mathcal{L}|},x^e(\overline{u})]$. This means that $\phi(T,0,x^*) = G_T(x^*) = x^*$ and $t\mapsto \phi(t,0,x^*)$ is a periodic orbit in $\mathcal{F}$. 
 
Regarding parts~\ref{p2:acyclic-period} and~\ref{p2:perodic-converge}, let $u^*\in \subscr{\mathcal{U}}{str}$ be such that $\overline{u}<u^*$. By Lemma~\ref{lem:monotone-eqpt}, we have $x^e(\overline{u})<x^e(u^*)$ and $\rho(t)\in \mathrm{int}\big([\vect{0}_{|\mathcal{L}|},x^e(u^*)]\big)$, for every $t\in \real_{\ge 0}$. Additionally, the box $[\vect{0}_{|\mathcal{L}|},x^e(u^*)]$ is inside $\mathcal{M}$ and is an invariant set for the system. By Proposition~\ref{thm:free-flow-property}\ref{p2:weak-contraction}, the system is monotone and weakly contracting with respect to the $\ell_1$-norm on $[\vect{0}_{|\mathcal{L}|},x^e(u^*)]$.  Moreover, by Assumption~\ref{asu:invertible}, for every $v\in V$, there exists a directed path of links $(i_1,\ldots,i_k)$ from $v$ to some $w\in \supscr{V}{out}$. Since $\rho(t)\in \mathcal{F}$, we have that $\supscr{f}{in}_i(\rho(t)) = R^{\tau(i)}_i\sum_{j\in\supscr{\mathcal{L}}{in}_{\tau(i)}} d_j(\rho(t))$, for every $i\in \mathcal{O}$. Note that $\rho(t)\in \mathrm{int}\big([\vect{0}_{|\mathcal{L}|},x^e(u^*)]\big)$, for every $t\in \real_{\ge 0}$, and $d_i$ is strictly increasing, for every $i\in \mathcal{L}$. This implies that $\frac{\partial}{\partial x_{i_{r}}}(\supscr{f}{in}_{i_{r+1}}(\rho(t)))>0$, for every $r\in \{0,\ldots,k-1\}$. Now, using~\cite[Proposition 2]{EL-GC-KS:14}, every trajectory of the system starting in the box $[\vect{0}_{|\mathcal{L}|},x^e(u^*)]$ converges to the periodic orbit $\rho$. This implies that $\rho$ is locally asymptotically stable. 

Regarding part~\ref{p2:perodic-converge}, consider the dynamical systems:
 \begin{align}
     \dot{x} &= F(x,u(t)),\qquad x\in \real^{|\mathcal{L}|},\label{eq:ds1}\\
     \dot{z} &= F(z,\overline{u}),\quad\qquad z\in \real^{|\mathcal{L}|}\label{eq:ds2}
 \end{align}
 where $F$ is as defined in~\eqref{eq:traffic} and~\eqref{eq:FIFO-traffic}. Let $x_y:\real_{\ge 0}\to \real^{|\mathcal{L}|}$ be the trajectory of the dynamical system~\eqref{eq:ds1} starting at $y$ and $z_y:\real_{\ge 0}\to \real^{|\mathcal{L}|}$ be the trajectory of the dynamical system~\eqref{eq:ds2} starting at $y$. Using definition of the vector field $F$ in~\eqref{eq:traffic} and~\eqref{eq:FIFO-traffic}, 
 \begin{align}\label{eq:ineq}
  F(x,u(t)) \le F(x,\overline{u}), \quad\mbox{ for }x\in \real^{|\mathcal{L}|}, t\in \real_{\ge 0}.  
 \end{align}
  In turn,~\cite[Lemma 3.8.1]{ANM-LH-DL:08-new} and inequality~\eqref{eq:ineq} imply that
 \begin{align}\label{eq:ineq-late}
     x_{y}(t) \le z_y(t)\qquad \mbox{ for every } t\in \real_{\ge 0}.
 \end{align}
 Note that, $y\in \mathcal{M}$ is a monotone-invariant point of the system. Therefore, by Theorem~\ref{thm:equilibrium}\ref{p4:free-flow}, we have $\lim_{t\to\infty} z_y(t)=x^e(\overline{u})<x^e(u^*)$. Moreover, using the inequality~\eqref{eq:ineq-late}, there exists $T\in \real_{\ge 0}$, such that $x_y(t) \in [\vect{0}_{|\mathcal{L}|},x^e(u^*)]$, for every $t\ge T$. Thus, using the proof of part~\ref{p2:acyclic-period}, the trajectory $t\to x_y(t)$ converges to the periodic orbit $\rho$. Note that the system is monotone on $\mathcal{M}$ and $[\vect{0}_{|\mathcal{L}|},y]\subseteq \mathcal{M}$. Therefore, for every $p\in [\vect{0}_{|\mathcal{L}|},y]$, we have $x_p(t)\le x_y(t)$, for every $t\in \real_{\ge 0}$ and thus $t\mapsto x_p(t)$ converges to $\rho$. \end{proof}

\begin{remark}[Comparison with the literature]
Existence and stability of periodic orbits for monotone dynamic flow networks with non-FIFO rules have been studied in~\cite{EL-GC-KS:14}. To our knowledge, Theorem~\ref{thm:orbit} is the first result that studies periodic orbits of the flow network \eqref{eq:traffic} with the FIFO rule~\eqref{eq:FIFO-traffic}.   
\end{remark}

    \section{Regions of attraction via monotone-invariant points}\label{sec:invariantpoint}
      Theorems~\ref{thm:equilibrium} and~\ref{thm:orbit} provide a framework for constructing inner-estimate of the regions of attraction of dynamic flow network. However, computing these inner-estimates relies on finding a suitable monotone-invariant point for the system. In this section, we provide two methods for finding monotone-invariant points of the flow network~\eqref{eq:traffic} and~\eqref{eq:FIFO-traffic}.
     \begin{proposition}[Monotone-invariant points via vector field]\label{prop:invariant2}
      Consider the dynamical flow network~\eqref{eq:traffic} with the
      FIFO rule~\eqref{eq:FIFO-traffic} and a feasible input metering $u:\real_{\ge 0}\to \mathcal{U}$. Suppose that $u$ satisfies Assumption~\ref{assum:good} with an upper-bound $\overline{u}\in \mathcal{U}$. Then, the following statement holds:
    \begin{enumerate}
    \item\label{p1:arbitrary} any point $y\in \mathcal{M}$ satisfying $F(y,\overline{u})\le \vect{0}_{|\mathcal{L}|}$ is  a monotone-invariant point of the system.
    \end{enumerate}
    In particular, we have
    \begin{enumerate}\setcounter{enumi}{1}
      \item\label{p2:freefloweqpt} $x^e(\overline{u})$ is a monotone-invariant point of the system;
      \item\label{p3:generaleqpt} any equilibrium point $x^*$ of the vector field $x\mapsto F(x,\overline{u})$ in the monotone-flow domain $\mathcal{M}$ is a monotone-invariant point of the system.
    \end{enumerate} 
    \end{proposition}
    \begin{proof}
    Note that $F(y,\overline{u})\le \vect{0}_{|\mathcal{L}|}$. Thus, by Lemma~\ref{prop:invariant}, the box $[\vect{0}_{|\mathcal{L}|},y]$ is an invariant set. Moreover, we have $y\in\mathcal{M}$. Therefore, by Proposition~\ref{thm:free-flow-property}\ref{p1:cone}, the box $[\vect{0}_{|\mathcal{L}|},y]$ is in $\mathcal{M}$. This implies that $y$ is a monotone-invariant point for the system. Parts~\ref{p2:freefloweqpt} and~\ref{p3:generaleqpt} easily follow from part~\ref{p1:arbitrary}.  
    \end{proof}

   \begin{remark}
   \begin{enumerate}
  \item  For a constant input $u\in \mathcal{U}$, one can use Proposition~\ref{prop:invariant2}\ref{p2:freefloweqpt} and Theorem~\ref{thm:equilibrium} to show that $\bigcup_{v\in \mathcal{U}, v\ge u} [\vect{0}_{|\mathcal{L}|},x^e(v)]$ is a region of attraction for the equilibrium point $x^e(u)$ of the flow network~\eqref{eq:traffic} and~\eqref{eq:FIFO-traffic}. This leads to larger inner-estimates of regions of attraction than the one proposed in \cite[Proposition 3]{SC-MA:15}.
  \item Proposition~\ref{prop:invariant2}\ref{p3:generaleqpt} does not require the equilibrium point $x^*$ to be in the free-flow domain $\mathcal{F}$. It only requires that the out-links of the diverging nodes at the equilibrium point $x^*$ are in free-flow.
  \end{enumerate}
    \end{remark}
    
    In the next proposition, we establish an algorithmic approach for finding monotone-invariant points. Our algorithm is based on the \emph{monotone-flow iteration} defined by
   \begin{align}\label{eq:km}
        x^{k+1} &= x^{k} + \alpha H(x^k,\overline{u}),\qquad\mbox{ for all }k\in \mathbb{Z}_{\ge 0},\nonumber\\
        x^0 &= \overline{x}, 
    \end{align}
    where $H$ is the monotone extension vector field in~\eqref{eq:H}. 
    \begin{proposition}[Monotone-invariant points via forward Euler iterations]\label{prop:iterations}
    Consider the dynamical flow network~\eqref{eq:traffic} and~\eqref{eq:FIFO-traffic} with the strictly feasible input metering $u:\real_{\ge 0}\to \subscr{\mathcal{U}}{str}$. Suppose that $u$ satisfies Assumption~\ref{assum:good} with the upper-bound $\overline{u}\in \subscr{\mathcal{U}}{str}$. Then, for small enough $\alpha>0$, and 
        \begin{align*}
            N=\max\{k\in \mathbb{Z}_{\ge 0}\mid x^k\not\in \mathrm{int}(\mathcal{M})\},
            \end{align*}
            the following statements hold:
    \begin{enumerate}
        \item\label{p1:converge-km} the monotone-flow iteration~\eqref{eq:km} converges to $x^e(\overline{u})$;
        \item\label{p1.5} $x^e(\overline{u})\le x^{N+1}$;
        \item\label{p2:eqpt-km} $x^{N+1}$ is a monotone-invariant point for the system.  
    \end{enumerate}
    \end{proposition}
    \begin{proof}
    Consider the following dynamical systems
    \begin{align}
        \dot{z} &= H(z,u(t)),&& z\in \real^{|\mathcal{L}|}\label{eq:original}\\
        \dot{x} &= H(x,\overline{u}),&& x\in \real^{|\mathcal{L}|}\label{eq:bound}
    \end{align}
     We first note that the monotone-flow iteration~\eqref{eq:km} is the forward Euler discretization of the solution of the dynamical system~\eqref{eq:bound} starting at $\overline{x}$. Let $x_y:\real_{\ge 0}\to \real^{|\mathcal{L}|}$ be the trajectory of system~\eqref{eq:bound} starting at $y$, $z_y:\real_{\ge 0}\to \real^{|\mathcal{L}|}$ be the trajectory of system~\eqref{eq:original} starting at $y$, and $\{x^k_{y}\}_{k=1}^{\infty}$ be the sequence generated by monotone-flow iteration~\eqref{eq:km} starting at $x^0=y$. By~\cite[Theorem 6.3]{KEA:08}, for every $\epsilon>0$, one can choose $\alpha\in \real_{\ge 0}$ small enough such that, for every $z\in [0,\overline{x}]$ and every $k\in \{1.\ldots,N+1\}$, we have
\begin{align}\label{eq:boundgood}
    \|x_z(k\alpha)-x_z^{k}\| \le \frac{\epsilon}{2}.
\end{align}
Regarding part~\ref{p1:converge-km}, we define the map $\mathcal{T}:\real^{|\mathcal{L}|}\to \real^{|\mathcal{L}|}$ by $\mathcal{T}(z) = z + \sqrt{\alpha} H(z,\overline{u})$. We first show that $\mathcal{T}$ is non-expansive. Note that, by~\cite[Lemma 7]{SJ-AD-AVP-FB:21f}, we have
\begin{align*}
   \|D_x\mathcal{T}(x)\|_1 & = \|I_n + \sqrt{\alpha} D_x H(x,\overline{u})\|_1 \\ & = 1 + \sqrt{\alpha} \mu_1(D_x H(x,\overline{u})),\qquad\mbox{ for all } x\in \real^{|\mathcal{L}|}.
\end{align*}
Moreover, by Proposition~\ref{thm:ext}\ref{p1:H}, $H$ is weakly contracting with respect to $\ell_1$-norm, i.e., we have $\mu_1(D_x H(x,\overline{u}))\le 0$, for every $x\in \real^{|\mathcal{L}|}$. As a result, we get $\|D_x\mathcal{T}(x)\|_1 \le 1$
and thus $\mathcal{T}$ is non-expansive with respect to $\ell_1$-norm. 
The \KM iteration for the operator $\mathcal{T}$ is given by
\begin{align*}
    x^{k+1} &= (1-\sqrt{\alpha}) x^k + \sqrt{\alpha}\mathcal{T}(x^k) \\ & = x^{k} + \alpha H(x^k,\overline{u}),\qquad\mbox{ for all }k\in \mathbb{Z}_{\ge 0}.
\end{align*}
Thus, the monotone-flow iteration~\eqref{eq:km} is the \KM iteration for the operator $\mathcal{T}$. By~\cite[Corollary 11]{JB-SR-IS:92}, either the monotone-flow iteration~\eqref{eq:km} grows unbounded or it converges to a fixed-point of the operator $\mathcal{T}$. One can use the bound~\eqref{eq:boundgood} and the fact that, every trajectory of the system~\eqref{eq:bound} converges to $x^e(\overline{u})$, to show that the monotone-flow iteration~\eqref{eq:km} is bounded and thus it converges to a fixed point of $\mathcal{T}$. 
Note that $x^*$ is a fixed points of $\mathcal{T}$ if and only if it is an equilibrium points of the system~\eqref{eq:bound}. Moreover, $H(x,\overline{u})$ is weakly contracting with respect to the $\ell_1$-norm and $x^e(\overline{u})$ is a locally asymptotically stable equilibrium point of~\eqref{eq:bound}. Therefore, by~\cite[Theorem 19]{SJ-PCV-FB:19q}, $x^e(\overline{u})$ is the unique equilibrium point of the system~\eqref{eq:bound}. As a result, $x^e(\overline{u})$ is the only fixed-point of the operator $\mathcal{T}$. Thus, the monotone-flow iteration~\eqref{eq:km} converges to $x^e(\overline{u})$.  

Regarding parts~\ref{p1.5} and~\ref{p2:eqpt-km}, consider the trajectory of the dynamical system~\eqref{eq:bound} starting at $\overline{x}$, i.e., $x_{\overline{x}}:\real_{\ge 0}\to \real^{|\mathcal{L}|}$. By~\cite[Theorem 19]{SJ-PCV-FB:19q}, one can show that $x_{\overline{x}}$ converges to $x^e(\overline{u})$. Therefore, we can define $T=\sup\{t\in\real_{\ge 0}\mid x_{\overline{x}}(t)\not\in \mathcal{M}\}$. For every $s>T$, the trajectory of the system~\eqref{eq:bound} starting at $x_{\overline{x}}(s)$ remains inside $\mathcal{M}$. Moreover, we have
\begin{align*}
H(x,u(t))\le H(x,\overline{u}),\qquad\mbox{ for every }x\in \real^{|\mathcal{L}|} .   
\end{align*}
Using~\cite[Lemma 3.8.1]{ANM-LH-DL:08-new}, for every $y\in \real^{|\mathcal{L}|}$, we have $z_y(t)\le x_y(t)$. Therefore, for every $s>T$, the trajectory of the system~\eqref{eq:original} starting at $x_{\overline{x}}(s)$ remains inside $\mathcal{M}$ and thus $x_{\overline{x}}(s)$ is a monotone-invariant point. The result then follows using a continuity argument and the bound~\eqref{eq:boundgood}. \end{proof}

In the next examples, we use our framework to estimate regions of attraction of free-flow equilibrium points for acyclic and cyclic flow networks.

    \begin{example}[An acyclic dynamic flow network]\label{ex:1}
     Consider the dynamic flow network~\eqref{eq:traffic} and~\eqref{eq:FIFO-traffic} with the topology shown in Figure~\ref{fig:example} and the demand and supply functions as below:
     \begin{align*}
         d_i(x_i) &= \min\{15,x_i\},&&\quad i\in\{1,4,5\}\\ s_i(x_i) &= \min\{15,30-x_i\},&&\quad i\in\{1,4,5\},\\
         d_i(x_i) &= \min\{50,x_i\},&&\quad i\in\{2,3\},\\ s_i(x_i) &= \min\{50,100-x_i\},&&\quad i\in\{2,3\}
     \end{align*}
     We assume that all the non-zero split ratios are equal to $1$ except  $R^{\sigma(1)}_2=R^{\sigma(1)}_3=\frac{1}{2}$ and we have a constant input metering $u=8$.
          \begin{figure}
              \centering
              \includegraphics[width=0.9\linewidth]{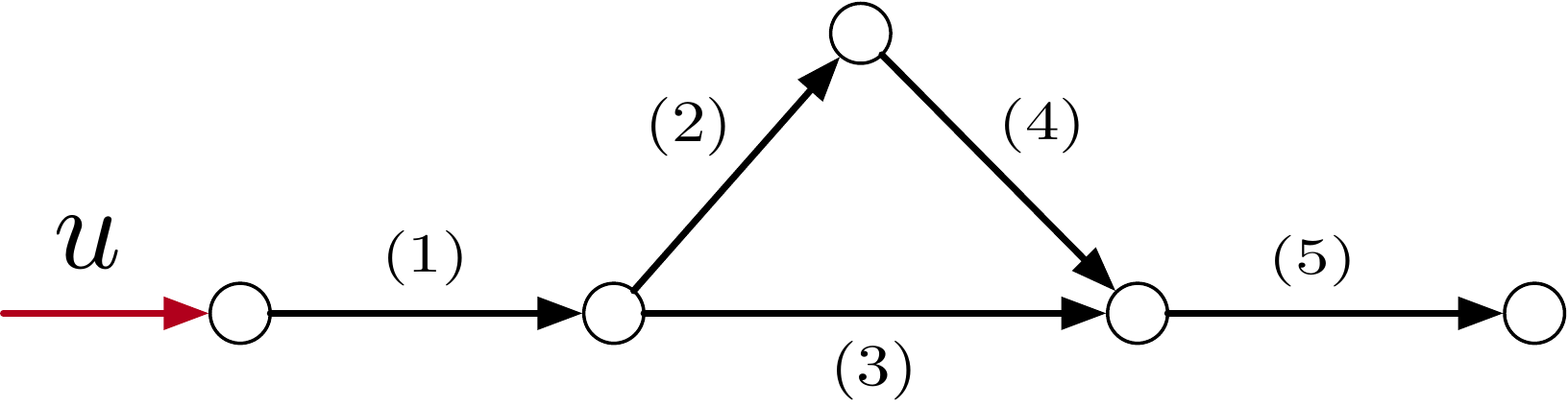}
              \caption{Topology of the flow network in Example~\ref{ex:1} with $\mathcal{R}=\{1\}$ and $\mathcal{O}=\{2,3,4,5\}$.}
              \label{fig:example}
            \end{figure}
        It is easy to compute the free-flow equilibrium point $x^e(u)=(8,4,4,4,8)^{\top}$. Note that $ \subscr{\mathcal{U}}{str}=[0,15)$. Therefore, by Theorem~\ref{thm:equilibrium}\ref{p4:free-flow} and Proposition~\ref{prop:invariant2}\ref{p2:freefloweqpt}, for every $\overline{u}\in [0,15)$, the box $[\vect{0}_5,x^e(\overline{u})]$ is in the region of attraction of  $x^e(u)=(8,4,4,4,8)^{\top}$. For $\overline{u}=15$, one can compute $x^e(15)=(15, 7.5, 7.5, 7.5, 15)^{\top}$. Thus, every trajectory of the system starting in the box $[\vect{0}_5, (15, 7.5, 7.5, 7.5, 15)^{\top}]$ converges to the free-flow equilibrium point $(8,4,4,4,8)^{\top}$. Note that, $\overline{u}=15$ is not a strictly feasible input metering. One can check that
          \begin{align*}
              x^* = (15, 92.5, 22.5, 22.5, 15)^{\top}
          \end{align*}
           is another equilibrium point (different from $x^e(15)$ and satisfying $x^e(15)\le x^*$) of the vector field $F(x,15)$ in the monotone-flow domain $\mathcal{M}$. Using Theorem~\ref{thm:equilibrium}\ref{p4:free-flow} and Proposition~\ref{prop:invariant2}\ref{p3:generaleqpt}, every trajectory starting in the box $[\vect{0}_5, (15, 92.5, 22.5, 22.5, 15)^{\top}]$ converges to $(8,4,4,4,8)^{\top}$. One can also use the monotone-flow iteration~\eqref{eq:km} to obtain a monotone-invariant point of the system. By setting $\alpha=0.01$ and $\overline{u}=14$, we obtain $N=108246$ and $x^{N+1}=(14.9,\;\;   92.55,\;\;  22.15,\;\;   22.45,\;\;  15)^{\top}$. Therefore, by Proposition~\ref{prop:iterations}, the point $(14.9,\;\;   92.55,\;\;  22.15,\;\;   22.45,\;\;  15)^{\top}$ is a monotone-invariant point of the system and, by Theorem~\ref{thm:equilibrium}\ref{p4:free-flow}, the box $[\vect{0}_5,(14.9,\;\;   92.55,\;\;   22.15,\;\;   22.45,\;\;  15)^{\top}]$ is in the region of attraction of the equilibrium point $(8,4,4,4,8)^{\top}$.
    \end{example}

\begin{example}[A cyclic dynamic flow network]\label{ex:2}
Consider the dynamic flow network~\eqref{eq:traffic} and~\eqref{eq:FIFO-traffic} with the topology shown in Figure~\ref{fig:example2} and the demand and supply functions as below:
     \begin{align*}
         d_i(x_i) &= \min\{15,x_i\},&&\quad i\in\{1,2,3,4\}\\ s_i(x_i) &= \min\{15,30-x_i\},&&\quad i\in\{1,2,3,4\}
     \end{align*}
     We assume that all the non-zero split ratios are equal to $1$ except  $R^{\sigma(3)}_2=R^{\sigma(3)}_4=\frac{1}{2}$ and we have a constant input metering $u=5$.
     \begin{figure}
              \centering
              \includegraphics[width=0.9\linewidth]{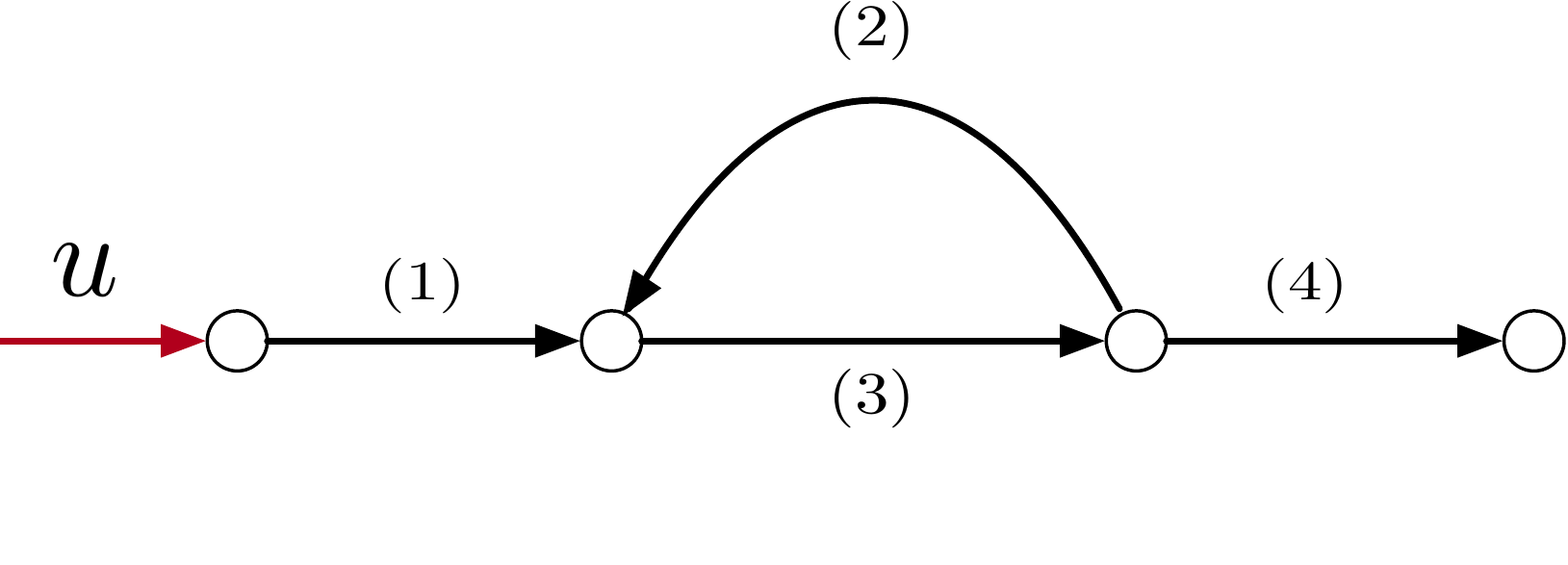}
              \caption{Topology of the flow network in Example~\ref{ex:2} with $\mathcal{R}=\{1\}$ and $\mathcal{O}=\{2,3,4\}$.}
              \label{fig:example2}
            \end{figure}
Note that $ \subscr{\mathcal{U}}{str}=[0,7.5)$ and one can check that the free-flow equilibrium point of the system is given by $x^e(u)=(5, 10, 5, 5)^{\top}$. Therefore, by Theorem~\ref{thm:equilibrium}\ref{p4:free-flow} and Proposition~\ref{prop:invariant2}\ref{p2:freefloweqpt}, for every $\overline{u}\in [0,7.5)$, the box $[\vect{0}_4,x^e(\overline{u})]$ is in the region of attraction of $(5, 10, 5, 5)^{\top}$. For $\overline{u}=7.5$, one can compute $x^e(7.5)=(7.5, 15, 7.5, 7.5)^{\top}$. A a result, every trajectory of the system starting in $[\vect{0}_4,(7.5, 15, 7.5, 7.5)^{\top}]$ converges to $(5, 10, 5, 5)^{\top}$. One can also use the monotone-flow iteration~\eqref{eq:km} to obtain a monotone-invariant point of the system. By setting $\alpha=0.01$ and $\overline{u}=7$, we obtain $N=2429$ and $x^{N+1} = (20.97,\;\; 15,\;\; 22.5,\;\; 7.5)^{\top}$. Therefore, by Proposition~\ref{prop:iterations}, the point $(20.97,\;\; 15,\;\; 22.5,\;\; 7.5)^{\top}$ is a monotone-invariant point of the system and, by Theorem~\ref{thm:equilibrium}\ref{p4:free-flow}, the box $[\vect{0}_4,(20.97,\;\; 15,\;\; 22.5,\;\; 7.5)^{\top}]$ is in the region of attraction of the equilibrium point $(5, 10, 5, 5)^{\top}$.
\end{example}

\section{CONCLUSIONS}
We study robustness of input metering with respect to transient disturbances in flow networks with FIFO rules. We use the notion of a monotone-invariant point to push the boundaries of applicability of monotone system theory in flow networks with FIFO rules. For this class of flow networks, we establish a framework for estimating regions of attraction of equilibrium points and periodic orbits.





\section*{APPENDIX}

\subsection{Two useful lemmas}
In this appendix, we present two lemmas which are crucial for the proof of the main results of this paper.

\begin{lemma}[Monotonicity of free-flow equilibrium points]\label{lem:monotone-eqpt}
Consider the dynamical flow network~\eqref{eq:traffic} and~\eqref{eq:FIFO-traffic}. Given $u,w\in \mathcal{U}$,
\begin{enumerate}
    \item\label{p1:mon} if $u\le w$, then $x^e(u)\le x^e(w)$;
    \item\label{p2:mon} if $u<w$, then $x^e(u)<x^e(w)$.
\end{enumerate}
\end{lemma}
\begin{proof}
Regarding part~\ref{p1:mon}, for $i\in \mathcal{R}$, we have $x^e_i(u)=d_i^{-1}(u_i)$. Since $d_i$ is strictly increasing and $u_i\le w_i$, we get that
\begin{align*}
  x^e_i(u) =  d_i^{-1}(u_i) \le  d_i^{-1}(w_i) = x^e_i(w). 
\end{align*}
For $i\in \mathcal{R}$, we have $x^e_{i}(u) = d_i^{-1}([Pu]_{i})$. Since $P\ge 0$ and $u\le w$, we have $Pu\le Pw$. Since $d_i$ is strictly increasing, we get $x^e_{i}(u) = d_i^{-1}([Pu]_{i}) \le  d_i^{-1}([Pw]_{i}) =x^e_{i}(w)$.
Regarding part~\ref{p2:mon}, we know that $x^e(u)\le x^e(w)$. Suppose that there exists $i\in \mathcal{L}$ such that $x^e_i(u)=x^e_i(w)$. Since $x^e(w)\in \mathcal{F}$, we have 
\begin{align*}
    \supscr{f}{out}_i(x^e(u)) = d_i(x^e_i(u)) = d_i(x^e_i(w))= \supscr{f}{out}_i(x^e(w)).
\end{align*}
Note that $x^e(w)$ satisfies $F(x^e(w),w)=\vect{0}_{|\mathcal{L}|}$. This implies that $\supscr{f}{in}_i(x^e(u),u) = \supscr{f}{in}_i(x^e(w),w)$. We now consider two cases. Suppose that $i\in \mathcal{R}$. In this case,
\begin{align*}
    u_i = \supscr{f}{in}_i(x^e(u),u) = \supscr{f}{in}_i(x^e(w),w) = w_i,
\end{align*}
which is a contradiction since we assumed that $u<w$. Now suppose that $i\in \mathcal{O}$. By Assumption~\ref{asu:invertible}, there exists a directed path from some node in $\supscr{V}{in}$ to $\sigma(i)$. Let $(i_0, i_1,\ldots,i_k=i)$ be the ordered set of links on this directed path. This means that, for every $r\in \{0,\ldots,k-1\}$, we have $i_{r}\in \supscr{\mathcal{L}}{in}_{\tau(i_{r+1})}$. For $v=\tau(i)$, we have
\begin{align}\label{eq:saberisverysmart}
    R^v_{i}\sum_{j\in \supscr{\mathcal{L}}{in}_v}d_j(x^e_j(u)) & = \supscr{f}{in}_i(x^e(u),u) = \supscr{f}{in}_i(x^e(w),w) \nonumber\\ & = R^v_i\sum_{j\in \supscr{\mathcal{L}}{in}_v}d_j(x^e_j(w)),
\end{align}
Note that $x^e(u)\le x^e(w)$ and the demand function $d_i$ is strictly increasing on $[0,\supscr{x}{crit}_i]$, for every $i\in \mathcal{L}$. This implies that $d_j(x^e(u))\le d_j(x^e(w))$, for every $j\in \mathcal{L}$. Combining this fact with equality~\eqref{eq:saberisverysmart} will lead to  
\begin{align*}
    d_j(x^e_j(u)) = d_j(x^e_j(w)),\quad\mbox{ for }j\in \supscr{\mathcal{L}}{in}_v,\; v=\tau(i).
\end{align*}
Using the fact that $i_{k-1}\in \supscr{\mathcal{L}}{in}_{\tau(i_k)}$, we get $x^e_{i_{k-1}}(u) = x^e_{i_{k-1}}(w)$. We can continue this procedure until we get to $x^e_{i_{0}}(u) = x^e_{i_{0}}(w)$. Since $i_0\in\mathcal{R}$ we get a contradiction by the previous case. This means that we have $x^e(u)<x^e(w)$.
\end{proof}

\begin{lemma}[Invariant sets]\label{prop:invariant}
      Consider the dynamical flow network~\eqref{eq:traffic} and~\eqref{eq:FIFO-traffic} with a feasible input metering $u:\real_{\ge 0}\to \mathcal{U}$. Suppose that $u$ satisfies Assumption~\ref{assum:good} with the upper-bound $\overline{u}$ and $y\in \mathcal{M}$ is such that $F(y,\overline{u})\le \vect{0}_{|\mathcal{L}|}$. Then the box $[\vect{0}_{|\mathcal{L}|},y]$ is an invariant set of the system.
      \end{lemma}
      \begin{proof}
Consider the dynamical systems:
 \begin{align}
     \dot{x} &= F(x,u(t)),\qquad x\in \real^{|\mathcal{L}|},\label{eq:ds1-lemma}\\
     \dot{z} &= F(z,\overline{u}),\quad\qquad x\in \real^{|\mathcal{L}|}\label{eq:ds2-lemma}
 \end{align}
 where $F$ is as defined in~\eqref{eq:traffic} and~\eqref{eq:FIFO-traffic}. Let $x_y:\real_{\ge 0}\to \real^{|\mathcal{L}|}$ be the trajectory of the dynamical system~\eqref{eq:ds1-lemma} starting at $y$ and $z_y:\real_{\ge 0}\to \real^{|\mathcal{L}|}$ be the trajectory of the dynamical system~\eqref{eq:ds2-lemma} starting at $y$. Since $y\in \mathcal{M}$, by Lemma~\ref{thm:free-flow-property}, the box $[\vect{0}_{|\mathcal{L}|},y]$ is inside $\mathcal{M}$ and both dynamical systems~\eqref{eq:ds1-lemma} and~\eqref{eq:ds2-lemma} are monotone on this box. Note that $F(y,\overline{u})\le \vect{0}_{|\mathcal{L}|}$. Therefore, using~\cite[Proposition 2.1]{HLS:95}, the trajectory $t\mapsto z_y(t)$ is non-increasing in $t$ and converges to an equilibrium point in $[\vect{0}_{|\mathcal{L}|},y]$. Using the definition of the vector field $F$ in~\eqref{eq:traffic} and~\eqref{eq:FIFO-traffic}, 
 \begin{align}\label{eq:ineq-F}
  F(x,u) \le F(x,\overline{u}), \qquad\mbox{ for every }x\in \real^{|\mathcal{L}|}.  
 \end{align}
  In turn,~\cite[Lemma 3.8.1]{ANM-LH-DL:08-new} and inequality~\eqref{eq:ineq-F} imply that, for every $t\in \real_{\ge 0}$, we have $x_{y}(t) \le z_y(t)$. This implies that $x_y(t)\in [\vect{0}_{|\mathcal{L}|},y]$, for every $t\in \real_{\ge 0}$.
 \end{proof}

\subsection{Proof of Proposition~\ref{thm:free-flow-property}}\label{app:monotone-flow}
  Regarding part~\ref{p1:cone}, note the definition of domain $\mathcal{F}$ in Definition~\eqref{def:monotoneflow} and the
  FIFO rule~\eqref{eq:FIFO-traffic} together implies that
  $\alpha^{v}(x)=1$, for every $v\in \supscr{V}{div}$. Note that $s_i$
  is decreasing and $d_i$ is increasing and $R^v_i$ are constant, for every $i\in \supscr{\mathcal{L}}{out}_v$. This implies that, if $y\le x$ then we have
  $\alpha^{v}(y)\le \alpha^{v}(x)$. By definition of
  $\alpha^{v}(y)$, this implies
  that $\alpha^{v}(y)=1$, for every $v\in \supscr{V}{div}$. As a result, using the FIFO
  rule~\eqref{eq:FIFO-traffic}, one can easily see that
  $\supscr{f}{out}_i(y)=d_i(y_i)$ for every $i\in \supscr{\mathcal{L}}{out}_v$ with $v\in \supscr{V}{div}$. This
  means that $y\in \mathcal{M}$. Regarding part~\ref{p2:weak-contraction}, the monotonicity of the dynamical~\eqref{eq:traffic} with the FIFO
  rule~\eqref{eq:FIFO-traffic} on $\mathcal{F}$ is clear since the terms $\supscr{f}{out}_i$ for $i\in \supscr{\mathcal{L}}{out}_v$ with $v\in \supscr{V}{div}$ were the only terms that cause non-monotonicity of the system and we show in part~\ref{p1:cone} that $\supscr{f}{out}_i(y)=d_i(y_i)$ for every $i\in \supscr{\mathcal{L}}{out}_v$ with $v\in \supscr{V}{div}$. For every $x\in \mathcal{M}$, we compute
  \begin{align*}
      \sum_{i=1}^{n} F_i(x,u) \le \sum_{i\in \mathcal{R}} s_i(x_i) - \sum_{v\in \supscr{V}{out}}\big(1-\sum_{i\in \supscr{\mathcal{L}}{out}_v} R^v_i\big)d_i(x_i).
  \end{align*}
  Since $x_i\mapsto s_i(x_i)$ is strictly decreasing and $x_i\mapsto d_i(x_i)$ is strictly increasing, for almost every $x\in \mathcal{M}$, 
   \begin{align*}
      \mu_1(D_x F(x,u)) = \max_j\frac{\partial }{\partial x_j}\sum_{i=1}^{n} F_i(x,u) \le 0
  \end{align*}
  where the first equality holds because $x\mapsto F(x,u)$ is monotone. This implies that $F$ is weakly contracting on $\mathcal{M}$. 
  
\subsection{Proof of Proposition~\ref{thm:ext}}\label{sec:ext}
  Regarding part~\ref{p1:H}, the monotonicity of the vector field $H$ follows from the following fact: the states $x$ for which $\supscr{f}{out}_i(x)\ne d_i(x_i)$ for $i\in \supscr{\mathcal{L}}{out}_v$ with $v\in \supscr{V}{div}$ are the only sources of non-monotonicity in the vector field $F$. However, for the vector field $H$, we have $\supscr{h}{out}_i(y)=d_i(y_i)$ for every $i\in \supscr{\mathcal{L}}{out}_v$ with $v\in \supscr{V}{div}$. To show that the system is weakly contracting with respect to $\ell_1$-norm, note that we have 
  \begin{align*}
      \sum_{i=1}^{n} H_i(x,u) \le \sum_{i\in \mathcal{R}} s_i(x_i) - \sum_{v\in \supscr{V}{out}}\big(1-\sum_{i\in \supscr{\mathcal{L}}{out}_v} R^v_i\big)d_i(x_i).
  \end{align*}
  Since $x_i\mapsto s_i(x_i)$ is strictly decreasing and $x_i\mapsto d_i(x_i)$ is strictly increasing, for almost every $x\in \real^{|\mathcal{L}|}$, 
   \begin{align*}
      \mu_1(D_x H(x,u)) =  \max_{j}\frac{\partial}{\partial x_j}\sum_{i=1}^{n}H_i(x,u) \le 0
  \end{align*}
  where the first equality holds because $x\mapsto H(x,u)$ is monotone.
  Regarding part~\ref{p2:H}, note that, for every $x\in \mathcal{M}$, we have $\supscr{f}{out}_i(x)=d_i(x_i)=\supscr{h}{out}_i(x)$, for every $i\in \supscr{\mathcal{L}}{out}_v$ with $v\in \supscr{V}{div}$. Thus, by comparing definition of $F$ in equations~\eqref{eq:traffic} and~\eqref{eq:FIFO-traffic} and definition of $H$ in~\eqref{eq:H}, it is clear that $F(x,u)=H(x,u)$, for every $x\in \mathcal{M}$.


\bibliographystyle{ieeetr}
\bibliography{alias,Main,FB,SJ}

\end{document}